\documentclass[12pt]{article}
\usepackage{amsfonts}

\newtheorem{theorem}{Theorem}

\newtheorem{corollary}[theorem]{Corollary}

\newtheorem{definition}[theorem]{Definition}
\newtheorem{example}[theorem]{Example}

\newtheorem{proposition}[theorem]{Proposition}
\newtheorem{remark}[theorem]{Remark}

\newenvironment{proof}[1][Proof]{\noindent\textbf{#1.} }{\ \rule{0.5em}{0.5em}}
\input{tcilatex}
\begin{document}

\title{Landweber-type operator and its properties}
\author{Andrzej Cegielski \and {\small Faculty Mathematics, Computer Science
and Econometrics } \and {\small University of Zielona G\'{o}ra, ul. Szafrana
4a, 65-516 Zielona G\'{o}ra, Poland, } \and {\small e-mail:
a.cegielski@wmie.uz.zgora.pl} \and {\small \medskip } \and {\small %
Department of Mathematics, Kuwait University, } \and {\small P.O. Box 5969,
Safat 13060, Kuwait}}
\maketitle

\begin{abstract}
Our aim is to present several properties of a Landweber operator and of a
Landweber-type operator. These operators are widely used in methods for
solving the split feasibility problem and the split common fixed point
problem. The presented properties can be used in proofs of convergence of
related algorithms.
\end{abstract}

\section{Preliminaries}

In this section we recall some notions and facts which we will use in the
further part of the paper. Let $\mathcal{H}$ be a real Hilbert space
equipped with an inner product $\langle \cdot ,\cdot \rangle $ and with the
corresponding norm $\Vert \cdot \Vert $. We say that an operator $S:\mathcal{%
H}\rightarrow \mathcal{H}$ is \textit{nonexpansive} (NE) if for all $x,y\in 
\mathcal{H}$ it holds $\left\Vert Sx-Sy\right\Vert \leq \left\Vert
x-y\right\Vert $. We say that $S$ is $\alpha $-\textit{averaged}, where $%
\alpha \in (0,1)$, if $S=(1-\alpha )\limfunc{Id}+\alpha N$ for a
nonexpansive operator $N$. We say that $S$ is \textit{firmly nonexpansive}
(FNE) if for all $x,y\in \mathcal{H}$ it holds 
\begin{equation}
\langle Sx-Sy,x-y\rangle \geq \left\Vert Sx-Sy\right\Vert ^{2}\text{.}
\label{e-FNE}
\end{equation}%
By $\limfunc{Id}$ we denote the \textit{identity} operator. If follows from (%
\ref{e-FNE}) that $S$ is FNE if and only if $\limfunc{Id}-S$ is FNE. Denote
by $S_{\lambda }:\limfunc{Id}+\lambda (S-\limfunc{Id})$ a $\lambda $\textit{%
-relaxation} of $S$, where $\lambda \geq 0$. $S$ is FNE if and only if $%
S_{\lambda }$ is NE for all $\lambda \in \lbrack 0,2]$. Moreover, $S$ is a $%
\lambda $-relaxation of an FNE operator if and only if $S$ is $\frac{\lambda 
}{2}$-averaged, $\lambda \in (0,2)$. Let $C\subseteq \mathcal{H}$ be a
nonempty closed convex subset. Then for any $x\in \mathcal{H}$ there is a
unique $y\in C$ satisfying $\left\Vert y-x\right\Vert \leq \left\Vert
z-x\right\Vert $ for all $z\in C$. This point is denoted by $P_{C}x$ and is
called the \textit{metric projection} of $x$ onto $C$. The metric projection 
$P_{C}$ is an FNE operator. By $\limfunc{Fix}S:=\{z\in \mathcal{H}:Sx=x\}$
we denote the subset of \textit{fixed points} of $S$. We say that an
operator $S:\mathcal{H}\rightarrow \mathcal{H}$ having a fixed point is $%
\alpha $-\textit{strongly quasi-nonexpansive} ($\alpha $-SQNE), where $%
\alpha \geq 0$, if for all $x\in \mathcal{H}$ and all $z\in \limfunc{Fix}S$
it holds 
\begin{equation}
\left\Vert Sx-z\right\Vert ^{2}\leq \left\Vert x-z\right\Vert ^{2}-\alpha
\left\Vert Sx-x\right\Vert ^{2}\text{.}  \label{e-SQNE}
\end{equation}%
If $\alpha >0$, then we call an $\alpha $-SQNE operator \textit{strongly
quasi-nonexpansive }(SQNE). A $0$-SQNE operator is called\textit{\
quasi-nonexpansive} (QNE). A $\lambda $-relaxation of an FNE operator having
a fixed point is $\frac{2-\lambda }{\lambda }$-SQNE, $\lambda \in (0,2]$. We
say that $S$ is a \textit{cutter} if 
\begin{equation}
\langle x-Sx,z-Sx\rangle \leq 0  \label{e-cutter}
\end{equation}%
for all $x\in \mathcal{H}$ and all $z\in \limfunc{Fix}S$. An operator $U:%
\mathcal{H}\rightarrow \mathcal{H}$ is $\alpha $-SQNE, if and only if 
\begin{equation}
\lambda (Ux-x,z-x\rangle \geq \left\Vert Ux-x\right\Vert ^{2}
\label{e-rel-cutter}
\end{equation}%
for all $x\in \mathcal{H}$ and all $z\in \limfunc{Fix}U$, where $\lambda =%
\frac{2}{\alpha +1}$. A FNE operator having a fixed point is a cutter. An
operator $S$ is a cutter if and only if $S_{\lambda }$ is $\frac{2-\lambda }{%
\lambda }$-SQNE, $\lambda \in (0,2]$. An extended collection of properties
of FNE as well as SQNE operators can be found, e.g., in \cite[Chapter 2]%
{Ceg12}.

\section{Landweber-type operator}

Let $\mathcal{H}_{1},\mathcal{H}_{2}$ be real Hilbert spaces, $A:\mathcal{H}%
_{1}\rightarrow \mathcal{H}_{2}$ be a bounded linear operator with $\Vert
A\Vert >0$, $C\subseteq \mathcal{H}_{1}$ and $Q\subseteq \mathcal{H}_{2}$ be
nonempty closed convex subsets.

\begin{definition}
\rm\ %
(\cite{Byr02}, \cite{Ceg14}) An operator $V:\mathcal{H}_{1}\rightarrow 
\mathcal{H}_{1}$ defined by 
\begin{equation}
V:=\limfunc{Id}+\frac{1}{\Vert A\Vert ^{2}}A^{\ast }(P_{Q}-\limfunc{Id})A
\label{e-LO}
\end{equation}%
is called a \textit{Landweber operator}. An operator $U:\mathcal{H}%
_{1}\rightarrow \mathcal{H}_{1}$ defined by 
\begin{equation}
U:=P_{C}(\limfunc{Id}+\frac{1}{\Vert A\Vert ^{2}}A^{\ast }(P_{Q}-\limfunc{Id}%
)A)  \label{e-PLO}
\end{equation}%
is called a \textit{projected Landweber operator}. An operator $R_{\lambda }:%
\mathcal{H}_{1}\rightarrow \mathcal{H}_{1}$ 
\begin{equation}
R_{\lambda }:=P_{C}(\limfunc{Id}+\frac{\lambda }{\Vert A\Vert ^{2}}A^{\ast
}(P_{Q}-\limfunc{Id})A)\text{,}  \label{e-PRLO}
\end{equation}%
where $\lambda \in (0,2)$, is called a \textit{projected relaxation} of a
Landweber operator $V$.
\end{definition}

A Landweber operator was applied by Landweber \cite{Lan51} in a method for
approximating least-squares solution of a first kind integral equation.
Censor and Elfving \cite{CE94} introduced the following problem: 
\begin{equation}
\text{find }x^{\ast }\in C\text{ with }Ax^{\ast }\in Q  \label{e-SFP}
\end{equation}%
and called it the \textit{split feasibility problem} (SFP). Byrne \cite%
{Byr02} applied a projected relaxation of a Landweber operator for solving
the SFP. Now we introduce a more general operator than the defined by (\ref%
{e-LO}).

\begin{definition}
\rm\ %
Let $T:\mathcal{H}_{2}\rightarrow \mathcal{H}_{2}$ be quasi-nonexpansive. An
operator $V:\mathcal{H}_{1}\rightarrow \mathcal{H}_{1}$ defined by 
\begin{equation}
V:=\limfunc{Id}+\frac{1}{\Vert A\Vert ^{2}}A^{\ast }(T-\limfunc{Id})A
\label{e-LtO}
\end{equation}%
is called a \textit{Landweber-type operator }(related to $T$).
\end{definition}

Because $P_{Q}$ is quasi-nonexpansive, (\ref{e-LO}) is a special case of (%
\ref{e-LtO}). The operators defined by (\ref{e-LO}), (\ref{e-PRLO}) and (\ref%
{e-LtO}) were applied by many authors for solving the split feasibility
problem, the multiple split feasibility problem and the split common fixed
point problem (see \cite{Byr02, Ceg14, CS09, Mou10, LMWX12, WX11, Xu10} and
the references therein).

\section{Properties of a Landweber-type operator}

Let $A:\mathcal{H}_{1}\rightarrow \mathcal{H}_{2}$ be a bounded linear
operator with $\Vert A\Vert >0$ and $Q\subseteq \mathcal{H}_{2}$ be closed
convex. By $\func{im}A$ we denote the image of $A$. The Landweber operator $%
V $ is closely related to a proximity function $f:\mathcal{H}_{1}\rightarrow 
\mathbb{R}
_{+}$ defined by the following equality 
\begin{equation}
f(x)=\frac{1}{2}\left\Vert P_{Q}(Ax)-Ax\right\Vert ^{2}\text{.}  \label{e-PF}
\end{equation}%
This relation is expressed in the following result, which proof can be found
in \cite[Proposition 2.1]{Byr02} or \cite[Lemma 4.6.2]{Ceg12}.

\begin{proposition}
\label{p-prox}The proximity function $f:\mathcal{H}_{1}\rightarrow 
\mathbb{R}
$ defined by (\ref{e-PF}) is a differentiable convex function and $%
Df=A^{\ast }(Ax-P_{Q}(Ax))$. Moreover, $\limfunc{Fix}V=\limfunc{Argmin}%
_{x\in \mathcal{H}_{1}}f$, where $V:=\limfunc{Id}+\frac{1}{\Vert A\Vert ^{2}}%
A^{\ast }(P_{Q}-\limfunc{Id})A$ denotes the Landweber operator.
\end{proposition}

A sequence $\{U^{k}x\}_{k=0}^{\infty }$, where $x\in \mathcal{H}$ is called
an \textit{orbit} of an operator $U:\mathcal{H}\rightarrow \mathcal{H}$.
Byrne \cite[Theorem 2.1]{Byr02} proved that any orbit of the a projected
relaxation $R_{\lambda }$ of a Landweber operator converges to a solution of
the SFP in the case $\mathcal{H}_{1}$ and $\mathcal{H}_{2}$ are Euclidean
spaces, $\lambda \in (0,2)$. Xu \cite{Xu06} observed that any orbit of $%
R_{\lambda }$ converges weakly in the infinite dimensional case.

Below we give an important property of a Landweber-type operator.

\begin{proposition}
\label{p-FNE}If $T:\mathcal{H}_{2}\rightarrow \mathcal{H}_{2}$ is firmly
nonexpansive, then a Landweber-type operator defined by (\ref{e-LtO}) is
firmly nonexpansive.
\end{proposition}

\begin{proof}
(cf. \cite[Theorem 4.6.3]{Ceg12}, where the case $T=P_{Q}$ was proved)
Recall that $U$ is FNE if and only if $\limfunc{Id}-U$ is FNE. Suppose now
that $T$ is FNE, i.e., 
\begin{equation}
\langle (u-Tu)-(v-Tv),u-v\rangle \geq \left\Vert (u-Tu)-(v-Tv)\right\Vert
^{2}  \label{e-IT}
\end{equation}%
for all $u,v\in \mathcal{H}_{1}$. Let $G:=\limfunc{Id}-V=\frac{1}{\left\Vert
A\right\Vert ^{2}}A^{\ast }(\limfunc{Id}-T)A$, where $V:\mathcal{H}%
_{1}\rightarrow \mathcal{H}_{1}$ is a Landweber-type operator given by (\ref%
{e-LtO}). We prove that $G$ is FNE. If we take $u:=Ax$ and $v:=Ay$ for $%
x,y\in \mathcal{H}_{1}$ in inequality (\ref{e-IT}) and apply the inequality $%
\left\Vert A^{\ast }z\right\Vert \leq \left\Vert A^{\ast }\right\Vert \cdot
\left\Vert z\right\Vert $ and the equality $\left\Vert A^{\ast }\right\Vert
=\left\Vert A\right\Vert $, then we obtain

\begin{eqnarray*}
\langle G(x)-G(y),x-y\rangle &=&\left\Vert A\right\Vert ^{-2}\langle A^{\ast
}(Ax-TAx)-A^{\ast }(Ay-TAy,x-y\rangle \\
&=&\left\Vert A\right\Vert ^{-2}\langle (Ax-TAx)-(Ay-TAy),Ax-Ay\rangle \\
&\geq &\left\Vert A\right\Vert ^{-2}\left\Vert (Ax-TAx)-(Ay-TAy)\right\Vert
^{2} \\
&=&\left\Vert A\right\Vert ^{-4}\left\Vert A^{\ast }\right\Vert
^{2}\left\Vert (Ax-TAx)-(Ay-TAy)\right\Vert ^{2} \\
&\geq &\left\Vert A\right\Vert ^{-4}\left\Vert A^{\ast }(Ax-TAx)-A^{\ast
}(Ay-TAy)\right\Vert ^{2} \\
&=&\Vert \left\Vert A\right\Vert ^{-2}A^{\ast }(Ax-TAx)-\left\Vert
A\right\Vert ^{-2}A^{\ast }(Ay-TAy)\Vert ^{2} \\
&=&\left\Vert G(x)-G(y)\right\Vert ^{2}\text{,}
\end{eqnarray*}%
i.e., $G$ is FNE. This yields that a Landweber-type operator $V=\limfunc{Id}%
-G$ is FNE.
\end{proof}

\begin{corollary}
If $T:\mathcal{H}_{2}\rightarrow \mathcal{H}_{2}$ is nonexpansive, then a
Landweber-type operator $V$ defined by (\ref{e-LtO}) is nonexpansive.
\end{corollary}

\begin{proof}
The Corollary follows easily from Proposition \ref{p-FNE} by an application
the fact that $S$ is FNE if and only if $2S-\limfunc{Id}$ is NE.
\end{proof}

\begin{definition}
\rm\ %
We say that an operator $S:\mathcal{H}\rightarrow \mathcal{H}$ is \textit{%
asymptotically regular }(AR) if 
\[
\lim_{k\rightarrow \infty }\Vert S^{k+1}x-S^{k}x\Vert =0 
\]%
for all $x\in \mathcal{H}$.
\end{definition}

Asymptotically regular nonexpansive operators play an important role in
fixed point iterations. An example of an asymptotically regular operator is
an SQNE one \cite[Theorem 3.4.3]{Ceg12}. Opial \cite[Theorem 1]{Opi67}
proved that any orbit of an AR and NE operator $S$ with $\limfunc{Fix}S\neq
\emptyset $ converges weakly to a fixed point of $S$. It turns out that the
same result one can obtain in a more general case. First we recall a notion
of the demi-closedness principle.

\begin{definition}
\rm\ %
We say that an operator $U:\mathcal{H}\rightarrow \mathcal{H}$ satisfies the 
\textit{demi-closedness} (DC) \textit{principle} if 
\begin{equation}
(x^{k}\rightharpoonup x\text{ and }\Vert Ux^{k}-x^{k}\Vert \rightarrow
0)\Longrightarrow x\in \limfunc{Fix}S\text{.}  \label{e-DC}
\end{equation}%
If (\ref{e-DC}) holds then we also say that $U-\limfunc{Id}$ is \textit{%
demi-closed} at $0$.
\end{definition}

Opial proved that a nonexpansive operator satisfies the DC principle \cite[%
Lemma 2]{Opi67}. Basing on this fact Opial proved in his famous result \cite[%
Theorem 1]{Opi67} that any orbit of a nonexpansive and asymptotically
regular operator $T$ having a fixed point converges weakly to $x^{\ast }\in 
\limfunc{Fix}T$. The demi-closedness principle also holds for a subgradient
projection $P_{f}$, for a continuous convex function $f:\mathcal{H}%
\rightarrow 
\mathbb{R}
$ with $S(f,0):=\{x\in \mathcal{H}:f(x)\leq 0\}\neq \emptyset $, which is
Lipschitz continuous on bounded subsets (see \cite[Theorem 4.2.7]{Ceg12}).

Now we present further properties of a Landweber-type operator. We apply
ideas of \cite[Section 2.4]{Ceg12} to a Landweber-type operator $V$. Recall
that an operator $S_{\tau }:\mathcal{H}_{1}\rightarrow \mathcal{H}_{1}$
defined by $S_{\sigma }x:=x+\tau (x)(Sx-x)$ is called a \textit{generalized
relaxation} of $S:\mathcal{H}_{1}\rightarrow \mathcal{H}_{1}$, where $\tau :%
\mathcal{H}_{1}\rightarrow (0,+\infty )$ is called a \textit{step-size
function} If $\tau (x)\geq 1$ for all $x\in \mathcal{H}_{1}$, then $S_{\tau
} $ is called an \textit{extrapolation }of $S$.

Let $V:=\limfunc{Id}+\frac{1}{\Vert A\Vert ^{2}}A^{\ast }(T-\limfunc{Id})A$
be a Landweber-type operator. Denote 
\begin{equation}
U:=\limfunc{Id}+A^{\ast }(T-\limfunc{Id})A\text{.}  \label{e-U}
\end{equation}%
Then, obviously, 
\begin{equation}
Ux-x=\Vert A\Vert ^{2}(Vx-x)\text{.}  \label{e-UV}
\end{equation}%
Let a step-size function $\sigma :\mathcal{H}_{1}\rightarrow (0,+\infty )$
be defined by 
\begin{equation}
\sigma (x):=\left\{ 
\begin{array}{ll}
\frac{\left\Vert TAx-Ax\right\Vert ^{2}}{\left\Vert A^{\ast
}(TAx-Ax)\right\Vert ^{2}}\text{,} & \text{if }Ax\notin \limfunc{Fix}T\text{,%
} \\ 
1\text{,} & \text{otherwise.}%
\end{array}%
\right.  \label{e-sigma}
\end{equation}%
An operator $U_{\sigma }$ defined by $U_{\sigma }x:=x+\sigma (x)A^{\ast
}(TAx-Ax)$ is a generalized relaxation of $U$. Denoting 
\begin{equation}
\tau (x)=\left\{ 
\begin{array}{ll}
\Vert A\Vert ^{2}\sigma (x)\text{,} & \text{if }Ax\notin \limfunc{Fix}T\text{%
,} \\ 
1\text{,} & \text{otherwise.}%
\end{array}%
\right.  \label{e-ro}
\end{equation}%
we obtain 
\begin{equation}
V_{\tau }x=\left\{ 
\begin{array}{ll}
x+\frac{\left\Vert TAx-Ax\right\Vert ^{2}}{\left\Vert A^{\ast
}(TAx-Ax)\right\Vert ^{2}}A^{\ast }(TAx-Ax)\text{,} & \text{if }Ax\notin 
\limfunc{Fix}T\text{,} \\ 
x\text{,} & \text{otherwise.}%
\end{array}%
\right.  \label{e-ext-Land}
\end{equation}%
By $\left\Vert A^{\ast }u\right\Vert \leq \left\Vert A^{\ast }\right\Vert
\cdot \left\Vert u\right\Vert =\left\Vert A\right\Vert \cdot \left\Vert
u\right\Vert $, we have $\sigma (x)\geq \left\Vert A\right\Vert ^{-2}$ and $%
\tau (x)\geq 1$. Therefore, $V_{\tau }$ is an extrapolation of the
Landweber-type operator $V$. Note, however, that we do not need to know the
norm of $A$ in order to evaluate $V_{\tau }x$. In the theorem below we
present important properties of $V_{\tau }$.

\begin{theorem}
\label{t-FixV}Let $A:\mathcal{H}_{1}\rightarrow \mathcal{H}_{2}$ be a
bounded linear operator with $\left\Vert A\right\Vert >0$, $T:\mathcal{H}%
_{2}\rightarrow \mathcal{H}_{2}$ be an $\alpha $-SQNE operator with $\func{im%
}A\cap \limfunc{Fix}T\neq \emptyset $, where $\alpha \geq 0$. Further, let
an extrapolation $V_{\tau }$ of a Landweber-type operator be defined by (\ref%
{e-ext-Land}). Then:

\begin{enumerate}
\item[$\mathrm{(i)}$] $\limfunc{Fix}V_{\tau }=\limfunc{Fix}V=A^{-1}(\limfunc{%
Fix}T)$;

\item[$\mathrm{(ii)}$] $V_{\tau }$ is $\alpha $-SQNE;

\item[$\mathrm{(iii)}$] If $\alpha >0$ then $V_{\tau }$ is asymptotically
regular;

\item[$\mathrm{(iv)}$] If $T$ satisfies the demi-closedness principle then $%
V_{\tau }$ also satisfies the demi-closedness principle.
\end{enumerate}
\end{theorem}

\begin{proof}
By (\ref{e-rel-cutter}), $T$ is $\alpha $-SQNE, where $\alpha \geq 0$, if
and only if 
\begin{equation}
\lambda \langle Tu-u,y-u\rangle \geq \left\Vert Tu-u\right\Vert ^{2}
\label{e-relC}
\end{equation}%
for all $u\in \mathcal{H}_{2}$ and all $y\in \limfunc{Fix}T$, where $\lambda
=2/(\alpha +1)\in (0,2]$. Note that $z\in A^{-1}(\limfunc{Fix}T)\ $if and
only if $Az\in \limfunc{Fix}T$.

(i) The equality $\limfunc{Fix}V_{\tau }=\limfunc{Fix}V$ is clear, because $%
\sigma (x)>0$ for all $x\in \mathcal{H}_{1}$. Now we prove the second
equality.

$\supseteq $ Let $Az\in \limfunc{Fix}T$. Then $Vz=z+(\frac{1}{\Vert A\Vert
^{2}}A^{\ast }(TAz-Az)=z$.

$\subseteq $ Let $z\in \limfunc{Fix}V$. Then, of course, $A^{\ast
}(TAz-Az)=0 $. Let $w\in \mathcal{H}_{1}$ be such that $Aw\in \limfunc{Fix}T$%
. By (\ref{e-relC}), we have 
\[
\left\Vert TAz-Az\right\Vert ^{2}\leq \lambda \langle TAz-Az,Aw-Az\rangle
=\lambda \langle A^{\ast }(TAz-Az),w-z\rangle =0\text{,} 
\]%
i.e., $Az\in \limfunc{Fix}T$.

(ii) Let $z\in \limfunc{Fix}V_{\sigma }=\limfunc{Fix}V$. By (i) $Az\in 
\limfunc{Fix}T$. (\ref{e-UV}) and (\ref{e-relC}) yield

\begin{eqnarray*}
&&\lambda \langle Vx-x,z-x\rangle \\
&=&\lambda \Vert A\Vert ^{-2}\langle Ux-x,z-x\rangle =\lambda \Vert A\Vert
^{-2}\langle A^{\ast }(TAx-Ax),z-x\rangle \\
&=&\lambda \Vert A\Vert ^{-2}\langle TAx-Ax,Az-Ax\rangle \geq \Vert A\Vert
^{-2}\cdot \left\Vert TAx-Ax\right\Vert ^{2} \\
&=&\Vert A\Vert ^{-2}\sigma (x)\left\Vert A^{\ast }(TAx-Ax)\right\Vert
^{2}=\Vert A\Vert ^{-2}\sigma (x)\left\Vert Ux-x\right\Vert ^{2} \\
&=&\tau (x)\left\Vert Vx-x\right\Vert ^{2}\text{,}
\end{eqnarray*}%
where $\lambda =2/(\alpha +1)\in (0,2]$, $U(x)$, $\sigma (x)$ and $\tau (x)$
are defined by (\ref{e-U}) and (\ref{e-sigma})--(\ref{e-ro}), $x\in \mathcal{%
H}_{1}$. This yields $\lambda \langle Vx-x,z-x\rangle \geq \tau
(x)\left\Vert Vx-x\right\Vert ^{2}$. Multiplying both sides by $\tau (x)$ we
obtain $\lambda \langle V_{\tau }x-x,z-x\rangle \geq \left\Vert V_{\tau
}x-x\right\Vert ^{2}$. By (\ref{e-rel-cutter}), the operator $V_{\sigma }$
is $\alpha $-SQNE.

(iii) Follows from (ii) and from the fact that any SQNE operator is AR (see 
\cite[Theorem 3.4.3]{Ceg12}).

(iv) Suppose, that $T$ satisfies the DC principle, i.e., $%
y^{k}\rightharpoonup y$ together with $\Vert Ty^{k}-y^{k}\Vert \rightarrow 0$
implies $y\in \limfunc{Fix}T$. We prove that $V_{\tau }$ also satisfies the
demi-closedness principle. Let $x^{k}\rightharpoonup x$ and $\Vert V_{\tau
}x^{k}-x^{k}\Vert \rightarrow 0$. The, obviously, $\left\Vert
Vx^{k}-x^{k}\right\Vert \rightarrow 0$, because $\tau (x^{k})\geq 1$.
Clearly, for any $u\in \mathcal{H}_{2}$, 
\[
\lim_{k}\langle Ax^{k}-Ax,u\rangle =\lim_{k}\langle x^{k}-x,A^{\ast
}u\rangle =0\text{,} 
\]%
i.e., $Ax^{k}\rightharpoonup Ax$. Choose an arbitrary $z\in \limfunc{Fix}%
V_{\tau }$. By (i), $Az\in \limfunc{Fix}T$. By (\ref{e-relC}), the
Cauchy--Schwarz inequality and the boundedness of $x^{k}$, we have 
\begin{eqnarray*}
\Vert TAx^{k}-Ax^{k}\Vert ^{2} &\leq &\lambda \langle
TAx^{k}-Ax^{k},Az-Ax^{k}\rangle \\
&=&\lambda \Vert A\Vert ^{2}\langle \frac{1}{\Vert A\Vert ^{2}}A^{\ast
}(TAx^{k}-Ax^{k}),z-x^{k}\rangle \\
&\leq &\lambda \Vert A\Vert ^{2}\Vert Vx^{k}-x^{k}\Vert \cdot \Vert
z-x^{k}\Vert \rightarrow 0\text{ as }k\rightarrow \infty \text{.}
\end{eqnarray*}%
Consequently, $\lim_{k}\Vert T(Ax^{k})-Ax^{k}\Vert =0$. This, together with $%
Ax^{k}\rightharpoonup Ax$ and with the assumption that $T$ satisfies the
demi-closedness principle, gives $Ax\in \limfunc{Fix}T$, i.e., $x\in A^{-1}(%
\limfunc{Fix}T)=\limfunc{Fix}V_{\tau }$ and the proof is completed.
\end{proof}

\bigskip

By an application of the inequality $\tau (x)\geq 1$, $x\in \mathcal{H}_{1}$%
, one can prove that Theorem \ref{t-FixV} implies the following result which
is a slight generalization of a result due to Wang and Xu \cite[Lemma 3.1
and Theorem 3.3]{WX11}.

\begin{corollary}
\label{l-T-cutter}Let $A:\mathcal{H}_{1}\rightarrow \mathcal{H}_{2}$ be a
bounded linear operator with $\left\Vert A\right\Vert >0$, $T:\mathcal{H}%
_{2}\rightarrow \mathcal{H}_{2}$ be an $\alpha $-SQNE operator with $\func{im%
}A\cap \limfunc{Fix}T\neq \emptyset $, where $\alpha \geq 0$. Further, let $%
V:=\limfunc{Id}+\frac{1}{\Vert A\Vert ^{2}}A^{\ast }(T-\limfunc{Id})A$. Then:

\begin{enumerate}
\item[$\mathrm{(i)}$] $\limfunc{Fix}V=A^{-1}(\limfunc{Fix}T)$ and

\item[$\mathrm{(ii)}$] $V$ is $\alpha $-SQNE.
\end{enumerate}

\noindent If, moreover, $T$ satisfies the DC principle, then $V$ also
satisfies the DC principle.
\end{corollary}

A proof of Corollary \ref{l-T-cutter} can be also found in \cite[Lemma 4.1]%
{Ceg14}. One can obtain similar results for a composition of an
extrapolation of a Landweber-type operator and an SQNE operator.

\begin{corollary}
\label{c-FixR}Let $T:\mathcal{H}_{2}\rightarrow \mathcal{H}_{2}$ $\alpha $%
-SQNE, $U:\mathcal{H}_{1}\rightarrow \mathcal{H}_{1}$ be $\beta $-SQNE with $%
\limfunc{Fix}U\cap A^{-1}(\limfunc{Fix}T)\neq \emptyset $, where $\alpha
,\beta >0$. Further, let $R:=U\circ V_{\tau }$, where $V_{\tau }$ is defined
by (\ref{e-ext-Land}). Then

\begin{enumerate}
\item[$\mathrm{(i)}$] $\limfunc{Fix}R=\limfunc{Fix}U\cap A^{-1}(\limfunc{Fix}%
T)$;

\item[$\mathrm{(ii)}$] $R$ is $\gamma $-SQNE, where $\gamma =(1/\alpha
+1/\beta )^{-1}$;

\item[$\mathrm{(iii)}$] $R$ is asymptotically regular.
\end{enumerate}

\noindent If, moreover, $T$ and $U$ satisfy the DC principle, then $R$ also
satisfies the DC principle.
\end{corollary}

\begin{proof}
By Theorem \ref{t-FixV} (i)-(ii) and by the assumption that $\limfunc{Fix}%
U\cap A^{-1}(\limfunc{Fix}T)\neq \emptyset $, $V_{\tau }$ and $U$ are SQNE
operators having a common fixed point. Therefore, (i) follows from \cite[%
Proposition 2.10 (i)]{BB96}. Part (ii) follows now from Theorem \ref{t-FixV}
(ii), from the assumption that $U$ is $\beta $-SQNE and from \cite[Theorem
2.1.48 (ii)]{Ceg12}. Part (iii) follows from (i), (ii) and from \cite[%
Theorem 3.4.3]{Ceg14}. Suppose now, that $T$ and $U$ satisfy the DS
principle. By Theorem \ref{t-FixV}, $V_{\tau }$ satisfies the DC principle.
Therefore, \cite[Theorem 4.2]{Ceg14} yields, that $R$ also satisfies the DC
principle.
\end{proof}

\bigskip

The results presented in Theorem \ref{t-FixV} and Corollary \ref{c-FixR} can
be applied to a proof of weak convergence of sequences generated by the
iteration below. Let $T:\mathcal{H}_{2}\rightarrow \mathcal{H}_{2}$ and $U:%
\mathcal{H}_{1}\rightarrow \mathcal{H}_{1}$ be QNE operators with $\limfunc{%
Fix}U\cap A^{-1}(\limfunc{Fix}T)\neq \emptyset $, $T_{k}:=T_{\lambda _{k}}$
and $U_{k}:=U_{\mu _{k}}$ be their relaxations, where the relaxation
parameters $\lambda _{k},\mu _{k}\in \lbrack \varepsilon ,1-\varepsilon ]$
for some small $\varepsilon >0$. Further, let 
\begin{equation}
V_{k}:=\left\{ 
\begin{array}{ll}
x+\frac{\left\Vert T_{k}Ax-Ax\right\Vert ^{2}}{\left\Vert A^{\ast
}(T_{k}Ax-Ax)\right\Vert ^{2}}A^{\ast }(T_{k}Ax-Ax)\text{,} & \text{if }%
Ax\notin \limfunc{Fix}T\text{,} \\ 
x\text{,} & \text{otherwise.}%
\end{array}%
\right.  \label{e-Vk}
\end{equation}%
be an extrapolation of a Landweber-type operator related to $T_{k}$.
Consider the following iterative process 
\begin{equation}
x^{k+1}=U_{k}V_{k}x^{k}\text{,}  \label{e-xk}
\end{equation}%
where $x^{0}\in \mathcal{H}_{1}$.

\begin{corollary}
Let $T:\mathcal{H}_{2}\rightarrow \mathcal{H}_{2}$ and $U:\mathcal{H}%
_{1}\rightarrow \mathcal{H}_{1}$ be QNE operators with $\limfunc{Fix}U\cap
A^{-1}(\limfunc{Fix}T)\neq \emptyset $ and satisfying the demi-closedness
principle. Then for any starting point $x^{0}\in \mathcal{H}_{1}$ the
sequence $\{x^{k}\}_{k=0}^{\infty }$ generated by (\ref{e-xk}) converges
weakly to a point $x^{\ast }\in \limfunc{Fix}U\cap A^{-1}(\limfunc{Fix}T)$.
\end{corollary}

\begin{proof}
It is easily seen that $U_{k}$ is $\alpha _{k}$-SQNE and $T_{k}$ is $\beta
_{k}$-SQNE, where $\alpha _{k}=\frac{1-\lambda _{k}}{\lambda _{k}}$, $\beta
_{k}=\frac{1-\mu _{k}}{\mu _{k}}$ and $\alpha _{k},\beta _{k}\in \lbrack
\varepsilon ,\frac{1}{\varepsilon }]$. By Corollary \ref{c-FixR} (i)-(ii), $%
\limfunc{Fix}(U_{k}\circ V_{k})=\limfunc{Fix}U\cap A^{-1}(\limfunc{Fix}T)$
and $U_{k}V_{k}$ is $\gamma _{k}$-SQNE with $\gamma _{k}=(\frac{1}{\alpha
_{k}}+\frac{1}{\beta _{k}})^{-1}\in \lbrack \frac{\varepsilon }{2},\frac{1}{%
2\varepsilon }]$. Therefore, the Corollary follows from a more general
result \cite[Theorem 5.1]{Ceg14}.
\end{proof}

\bigskip

\begin{remark}
\rm\ %
A special case of iteration (\ref{e-xk}) was recently presented by L\'{o}pez 
\textit{et al.} \cite{LMWX12}, where $T_{k}$ and $U_{k}$ are subgradient
projections related to subdifferentiable convex functions $c$ and $q$,
respectively, with $C:=\{x\in \mathcal{H}_{1}:c(x)\leq 0\}$ and $Q:=\{y\in 
\mathcal{H}_{2}:q(y)\leq 0\}$. L\'{o}pez \textit{et al.} \cite[Theorem 4.3]%
{LMWX12} proved the weak convergence of $x^{k}$ to a solution of the SFP (%
\ref{e-SFP}). Very recently, Cui and Wang \cite{CW14} studied the split
fixed point problem: find $x\in \limfunc{Fix}V$ with $Ax\in \limfunc{Fix}T$,
where $V:\mathcal{H}_{1}\rightarrow \mathcal{H}_{1}$ and $T:\mathcal{H}%
_{2}\rightarrow \mathcal{H}_{2}$ are $\beta $-demicontractive operators.
Recall that and operator $S:\mathcal{H}\rightarrow \mathcal{H}$ having a
fixed point is called $\beta $-contractive if 
\[
\left\Vert Tx-z\right\Vert ^{2}\leq \left\Vert x-z\right\Vert ^{2}+\beta
\left\Vert Tx-x\right\Vert ^{2}\text{,} 
\]%
for all $x\in \mathcal{H}_{1}$ and $z\in \limfunc{Fix}T$ (cf. \cite{Mou10}).
Cui and Wang applied a generalized relaxation of the operator $U$ defined by
(\ref{e-U}) for solving this problem. Roughly spoken, Cui and Wang \cite[%
Theorem 3.3]{CW14} observed that any orbit of the operator $V_{\lambda
}\circ U_{\rho }$, where the step-size $\rho (x)=\frac{1-\beta }{2}\sigma
(x) $ and the relaxation parameter $\lambda \in (0,1-\beta )$, converges
weakly to a solution of the problem. Their result is interesting, however,
an application of demicontractive operators seems to be artificial, because $%
T_{\frac{1-\beta }{2}}$ is a cutter, consequently $U_{\rho }$ is a cutter,
and $V_{\lambda }$ is strongly quasi-nonexpansive.
\end{remark}

Before we present the next property of a Landweber-type operator, we recall
a notion of an approximately shrinking operator.

\begin{definition}
\rm\ %
(\cite[Definition 3.1]{CZ14}) Let $U:\mathcal{H}\rightarrow \mathcal{H}$ be
quasi-nonexpansive. We say that $U$ is \textit{approximately shrinking }(AS)
if for any bounded subset $D\subseteq \mathcal{H}$ and for any $\eta >0$
there is $\gamma >0$ such that for any $x\in D$ it holds 
\begin{equation}
\Vert Ux-x\Vert <\gamma \Longrightarrow d(x,\func{Fix}U)<\eta \text{.}
\label{e-BS}
\end{equation}
\end{definition}

Important examples of AS operators are the metric projection $P_{C}$ onto a
closed convex subset $C\subseteq \mathcal{H}$ and the subgradient projection 
$P_{f}$ for a convex function $f:%
\mathbb{R}
^{n}\rightarrow 
\mathbb{R}
$ \cite[Lemma 24]{CZ13}.

\begin{theorem}
\label{t-AS}Let $T:%
\mathbb{R}
^{m}\rightarrow 
\mathbb{R}
^{m}$ be a QNE operator with $\func{im}A\cap \limfunc{Fix}T\neq \emptyset $.
Further, let $V:=\limfunc{Id}+\frac{1}{\Vert A\Vert ^{2}}A^{\ast }(T-%
\limfunc{Id})A$ be a Landweber-type operator and $V_{\tau }$ be its
extrapolation, where $\tau $ is given by (\ref{e-ro}). If $T$ is
approximately shrinking then $V$ and $V_{\tau }$ are also approximately
shrinking.
\end{theorem}

The theorem will be proved elsewhere. It turns out that the properties
presented in Theorems \ref{t-FixV} and \ref{t-AS} and in Corollary \ref%
{c-FixR} can be applied to a proof of the strong convergence of sequences
generated by a hybrid steepest descent method with an application of a
Landweber-type operators to a solution of a variational inequality (see \cite%
{CZ13} and \cite{CZ14} for the details).

\section{Examples}

\begin{example}
\rm\ %
Let $a\in \mathcal{H}$ with $\left\Vert a\right\Vert >0$, $A:\mathcal{H}%
\rightarrow 
\mathbb{R}
$ be defined by $Ax=\langle a,x\rangle $. Then $\left\Vert A\right\Vert
=\left\Vert a\right\Vert $. If $Q:=(-\infty ,\beta ],$ where $\beta \in 
\mathbb{R}
$, then the Landweber operator $V$ and its extrapolation $V_{\tau }$ with a
step-size function $\tau $ defined by (\ref{e-ro}) coincide and 
\[
Vx=V_{\tau }x=P_{C}x=x+\frac{(\langle a,x\rangle -\beta )_{+}}{\left\Vert
a\right\Vert ^{2}}a\text{,} 
\]%
where $C=H(a,\beta )_{-}:=\{u\in \mathcal{H}:\langle a,u\rangle \leq \beta
\} $ and $\alpha _{+}:=\max \{\alpha ,0\}$.
\end{example}

\begin{example}
\rm\ %
Let $A:\mathcal{H}_{1}\rightarrow \mathcal{H}_{2}$ be a bounded linear
operator with $\left\Vert A\right\Vert >0$. If $Q=\{b\}$, where $b\in 
\mathcal{H}_{2}$, then the SFP is to find a solution of a linear system $%
Ax=b $. The Landweber operator $V$ related to this problem and its
extrapolation $V_{\tau }$ with a step-size function $\tau $ defined by (\ref%
{e-ro}) have the forms 
\[
Vx=x-\frac{1}{\left\Vert A\right\Vert ^{2}}A^{\ast }(Ax-b) 
\]%
and 
\[
V_{\tau }x=x-\frac{\left\Vert Ax-b\right\Vert ^{2}}{\left\Vert A^{\ast
}(Ax-b)\right\Vert ^{2}}A^{\ast }(Ax-b)\text{.} 
\]%
By Theorem \ref{t-FixV} and Corollary \ref{l-T-cutter}, $V$ and $V_{\tau }$
satisfy the DC principle.
\end{example}

\begin{example}
\rm\ %
Let $A$ be an $m\times n$ real matrix representing a linear operator and $%
b\in 
\mathbb{R}
^{m}$. Suppose without loss of generality that the rows $a_{i}$ of $A$, $%
i=1,2,...,m$, are nonzero vectors. If $Q:=\{u\in \mathcal{%
\mathbb{R}
}^{m}:u\leq b\}=b-%
\mathbb{R}
_{+}^{m}$, then the SFP is to find a solution of a system of linear
inequalities $Ax\leq b$. The Landweber operator $V$ related to this problem
and its extrapolation $V_{\tau }$ with a step-size function $\tau $ defined
by (\ref{e-ro}) have the forms 
\[
Vx=x-\frac{1}{\lambda _{\max }(A^{T}A)}A^{T}(Ax-b)_{+} 
\]%
and 
\[
V_{\tau }x=x-\frac{\left\Vert (Ax-b)_{+}\right\Vert ^{2}}{\left\Vert A^{\ast
}(Ax-b)_{+}\right\Vert ^{2}}A^{T}(Ax-b)_{+}\text{,} 
\]%
where $A^{T}$ denotes the transposed matrix and $\lambda _{\max }(A^{T}A)$
denotes the maximal eigenvalue of $A^{T}A$. By Theorem \ref{t-FixV} and
Corollary \ref{l-T-cutter}, $V$ and $V_{\tau }$ satisfy the DC principle.
Moreover, by Theorem \ref{t-AS}, $V$ and $V_{\tau }$ are approximately
shrinking.
\end{example}

\begin{example}
\rm\ %
Let $f:\mathcal{H}_{2}\rightarrow 
\mathbb{R}
$ be a continuous convex function with $S(f,0):=\{y\in \mathcal{H}%
_{2}:f(y)\leq 0\}\neq \emptyset $. Define a subgradient projection $P_{f}:%
\mathcal{H}_{2}\rightarrow \mathcal{H}_{2}$, related to $f$ by 
\[
P_{f}(y)=\left\{ 
\begin{array}{ll}
y-\frac{f(y)_{+}}{\Vert g_{f}(y)\Vert ^{2}}g_{f}(y)\text{,} & \text{if }%
f(y)>0\text{,} \\ 
y\text{,} & \text{otherwise,}%
\end{array}%
\right. 
\]%
where $g_{f}(y)$ denotes a subgradient of $f$ at $y\in \mathcal{H}_{2}$.
Suppose that $f$ is Lipschitz continuous on bounded subsets. Let $A:\mathcal{%
H}_{1}\rightarrow \mathcal{H}_{2}$ be a bounded linear operator with $%
\left\Vert A\right\Vert >0$, $\lambda \in (0,2]$ and $T:=P_{f,\lambda }$.
Then $T$ is $\alpha $-SQNE with $\alpha =\frac{2-\lambda }{\lambda }$. A
Landweber operator $V$ related to this problem and its extrapolation $%
V_{\tau }$ with a step-size function $\tau $ defined by (\ref{e-ro}) have
the forms 
\[
Vx=\left\{ 
\begin{array}{ll}
x-\frac{\lambda f(Ax)_{+}}{\left( \Vert A\Vert \cdot \Vert g_{f}(Ax)\Vert
\right) ^{2}}A^{\ast }g_{f}(Ax)\text{,} & \text{if }f(Ax)>0\text{,} \\ 
x\text{,} & \text{otherwise}%
\end{array}%
\right. 
\]%
and 
\[
V_{\tau }x=\left\{ 
\begin{array}{ll}
x-\frac{\lambda f(Ax)_{+}}{\Vert A^{\ast }g_{f}(Ax)\Vert ^{2}}A^{\ast
}g_{f}(Ax)\text{,} & \text{if }f(Ax)>0\text{,} \\ 
x\text{,} & \text{otherwise.}%
\end{array}%
\right. 
\]%
Because $P_{f}$ satisfies the DC principle (see \cite[Theorem 4.2.7]{Ceg12}%
), Theorem \ref{t-FixV} and Corollary \ref{l-T-cutter} yield, that $V$ and $%
V_{\tau }$ satisfy the DC principle. Moreover, if $\mathcal{H}_{2}$ is
finite-dimensional, then $P_{f}$ is approximately shrinking (see \cite[Lemma
24]{CZ13}). Therefore, Theorem \ref{t-AS} yields that in this case $V$ and $%
V_{\tau }$ are also approximately shrinking.
\end{example}

\end{document}